\documentclass[a4paper, reqno, 11pt]{amsart}

\usepackage[usenames,dvipsnames]{color}

\usepackage{amsthm,amsfonts,amssymb,amsmath,amsxtra,amsrefs}
\usepackage[all]{xy}
\SelectTips{cm}{}
\usepackage{xr-hyper}
\usepackage[colorlinks=
   citecolor=Black,
   linkcolor=Red,
   urlcolor=Blue]{hyperref}
\usepackage{verbatim}
\makeatletter
\newtheorem*{rep@theorem}{\rep@title}
\newcommand{\newreptheorem}[2]{%
\newenvironment{rep#1}[1]{%
 \def\rep@title{#2 \ref{##1}}%
 \begin{rep@theorem}}%
 {\end{rep@theorem}}}
\makeatother

\usepackage[margin=1in]{geometry}

\usepackage{tikz}

\usepackage{mathrsfs}

\RequirePackage{xspace}
\RequirePackage{etoolbox}
\RequirePackage{varwidth}
\RequirePackage{enumitem}
\RequirePackage{tensor}
\RequirePackage{mathtools}
\RequirePackage{longtable}
\RequirePackage{multirow}

\newtheorem{theorem}{Theorem}
\newtheorem{prop}[theorem]{Proposition}

\newtheorem{lemma}[theorem]{Lemma}
\newtheorem{conj}[theorem]{Conjecture}
\newtheorem{cor}[theorem]{Corollary}

\theoremstyle{definition}

\newtheorem{question}[theorem]{Question}
\newreptheorem{theorem}{Theorem}

\newcommand{\cupdot}{\mathbin{\mathaccent\cdot\cup}}

\begin{document}

\title{Inverse problems for minimal complements \linebreak and maximal supplements} 
\author{Noga Alon}
\address{Noga Alon: Department of Mathematics, Princeton University, Princeton, NJ 08544, USA and Schools of Mathematics and Computer Science, Tel Aviv University, Tel Aviv 69978, Israel}
\email{nalon@math.princeton.edu}
\author{Noah Kravitz}
\address{Noah Kravitz: Grace Hopper College, Zoom University at Yale, New Haven, CT 06511}
\email{noah.kravitz@yale.edu}
\author{Matt Larson}
\address{Matt Larson: Department of Mathematics, Stanford, 450 Jane Stanford Way, Stanford, CA 94305}
\email{mwlarson@stanford.edu}

\date{\today}

\begin{abstract}
Given a subset $W$ of an abelian group $G$, a subset $C$ is called an additive complement for $W$ if $W+C=G$; if, moreover, no proper subset of $C$ has this property, then we say that $C$ is a minimal complement for $W$.  It is natural to ask which subsets $C$ can arise as minimal complements for some $W$.  We show that in a finite abelian group $G$, every non-empty subset $C$ of size $|C| \leq 2^{2/3}|G|^{1/3}/((3e \log |G|)^{2/3}$ is a minimal complement for some $W$.  As a corollary, we deduce that every finite non-empty subset of an infinite abelian group is a minimal complement.  We also derive several analogous results for ``dual'' problems about maximal supplements.
\end{abstract}

\maketitle

\section{Introduction}

The \emph{Minkowski sum} of two subsets $A,B$ of an (additive) abelian group $G$ is given by
$$A+B=\{a+b: a \in A, b \in B\}.$$
Given a subset $W \subseteq G$, we say that $C \subseteq G$ is an \emph{(additive) complement} for $W$ in $G$ if $W+C=G$ (that is, if every $g \in G$ can be expressed at least once as $g=w+c$, for $w \in W$, $c \in C$).  If, moreover, no proper subset $C' \subset C$ is a complement for $W$, then we say that $C$ is a \emph{minimal (additive) complement} for $W$ in $G$.  Similarly, we say that $C \subseteq G$ is an \emph{(additive) supplement} for $W$ in $G$ if the translations of $W$ by elements of $C$ are pairwise disjoint (that is, if every $g \in G$ can be expressed at most once as $g=w+c$, for $w \in W$, $c \in C$).  If, moreover, no proper superset $C' \supset C$ is a supplement for $W$, then we say that $C$ is a \emph{maximal (additive) supplement} for $W$ in $G$.  The project of this paper is to investigate which sets $C$ arise as minimal complements and maximal supplements not for a specific subset $W$ but rather for some $W$.  We study these questions in both finite and infinite abelian groups.

\subsection{Background}

The study of minimal complements in infinite abelian groups was initiated by Nathanson \cite{Nathanson}, who showed that if $W$ is a finite non-empty subset of $\mathbb{Z}$, then every complement $C$ for $W$ contains a minimal complement $C'$ for $W$.  The question of determining which subsets $W$ of $\mathbb{Z}$ admit minimal complements has received considerable attention. Chen and Yang \cite{Chen} showed that $W \subseteq \mathbb{Z}$ has a minimal complement if it is unbounded both above and below.  It is easy to check that, for instance, $\mathbb{N}$ does not have a minimal complement in $\mathbb{Z}$. Kiss, S\'{a}ndor, and Yang \cite{Kiss} studied the existence of minimal complements for ``eventually periodic'' subsets of $\mathbb{Z}$.  There has been some progress in infinite abelian groups other than $\mathbb{Z}$: Biswas and Saha \cite{Biswas3} generalized Nathanson's result by showing that if $G$ is any abelian group and $W$ is a finite non-empty subset of $G$, then any complement to $W$ contains a minimal complement.  The same authors \cite{Biswas2} later studied minimal complements in $\mathbb{Z}^r$ more closely.

The study of minimal complements and maximal supplements is related to problems in coding theory. Indeed, when $W$ is a ball in a Hamming space, the complements for $W$ correspond to covering codes. The possible sizes of such covering codes have been extensively studied in various settings (see, e.g., \cite{Honkala1991}). Similarly, maximal supplements for $W$ can be interpreted as sphere packings and thus are related to results such as Hamming bounds on block codes.

We remark that a 1995 paper of Habsieger and Ruzsa \cite{habsieger} treated the closely related matter of finding $C \subset \mathbb{Z}$ such that $W+C$ contains a given interval $[m]=\{1, \ldots, m\}$.  Additive bases and asymptotic additive bases, the study of which dates back at least to Lagrange's Four-Square Theorem, are of course also related to additive complements.  Objects similar to maximal supplements have been studied in the context of Sidon sets (see, e.g., \cite{ruzsa} and the references therein).

The natural ``inverse problem'' for minimal complements is determining whether a given set $C \subseteq G$ is a minimal complement for some set $W$ (in which case we say simply that $C$ is a \emph{minimal complement} in $G$).  This study was initiated by Kwon \cite{Kwon}, who showed that every finite non-empty subset of $\mathbb{Z}$ is a minimal complement. Biswas and Saha \cite{Biswas1} generalized this result to $\mathbb{Z}^r$.

In infinite groups, the inverse problem for minimal complements can behave quite differently from the non-inverse problem.  Nathanson \cite{Nathanson} showed that if $W$ is a finite subset of $\mathbb{Z}$, then every complement $C$ for $W$ contains a minimal complement $C' \subseteq C$ for $W$.  The same is not true, however, for the inverse problem: for instance $C=\{0,1,2\}$ is a complement for $W=2\mathbb{Z}$, but there is no $W' \subseteq 2\mathbb{Z}$ for which $C$ is a minimal complement.  (Of course, $C$ is a minimal complement for other sets, such as $3\mathbb{Z}$.)

\subsection{Main results}

In Section~\ref{sec:preliminaries}, we gather several basic facts about minimal complements, including that in a finite group $G$, no subset $C \subset G$ of size $2|G|/3<|C|<|G|$ can be a minimal complement.  We give a  description of the minimal complements in $G$ with size greater than $3|G|/5$ and the minimal complements that are arithmetic progressions.

In Section~\ref{sec:small}, we show with a probabilistic argument that every ``small'' subset of a finite group is a minimal complement.  Given a finite abelian group $G$, let $T(G)$ be the greatest integer $T$ such that every subset $C \subseteq G$ with size at most $T$ is a minimal complement.  Moreover, given a natural number $n$, let $T(n)$ be the minimum value of $T(G)$ as $G$ ranges over all abelian groups of order $n$.  In other words, $T(n)$ is the greatest integer $T$ such that for every group $G$ of order $n$, every subset $C \subseteq G$ with size at most $T$ is a minimal complement.
\begin{theorem}\label{thm:finite}
Let $G$ be a group of order $n$, and let $C \subset G$ be subset of size $|C|=k$. If $$0 < k \leq  
\frac{2^{2/3}n^{1/3}}{(3 e\log n)^{2/3}},$$ then $C$ is a minimal complement in $G$.  In other words, $T(n) \ge \frac{n^{1/3}}{2 (\log n/ \log 2)^{2/3}}$.
\end{theorem}

Since this growth rate for $T(n)$ can likely be improved, we have not optimized the constant in this theorem. We remark that this bound (like many in the paper) uses the estimate $|C-C| \leq k^2$; when we know that $C-C$ is small, we can often obtain slightly better dependence of $n$ on $k$.  Next, we use this theorem to obtain a generalization of the result of Biswas and Saha \cite{Biswas2} that every nonempty finite subset of $\mathbb{Z}^r$ is a minimal complement. 

\begin{theorem}\label{thm:infinite}
Let $G$ be an infinite abelian group, and let $C$ be a finite non-empty subset of $G$.
Then $C$ is a minimal complement in $G$. 
\end{theorem}

In Section~\ref{sec:upper}, we investigate upper bounds on $T(n)$.  By examining groups $G$ with subgroups of certain sizes, we derive (as a corollary to Proposition~\ref{prop:lower-construction}) that
$$\liminf_{n \to \infty} \frac{T(n)}{\sqrt{n}} \leq \sqrt{2}.$$
We also obtain an upper bound on $T(G)$ for all finite groups $G$.
\begin{theorem}\label{thm:upperbound}
For every $\varepsilon > 0$, we have that $T(G) = O(|G|^{3/4 + \varepsilon})$. In particular, we have that $T(n) = O(n^{3/4 + \varepsilon})$.
\end{theorem}

The proof of Theorem~\ref{thm:upperbound} goes by showing that, with high probability, a random subset formed by including independently each element of $G$ with probability $p = \vert G \vert^{\varepsilon - 1/4}$ 
is not a minimal complement. The proof of Theorem~\ref{thm:upperbound} 
works for any $p > \vert G \vert^{\varepsilon - 1/4}$ that is bounded away from $1$. In particular, taking $p = 1/2$, we see that a subset chosen uniformly at random is with high probability not a minimal complement.

Say that a subset of an abelian group $G$ is \emph{$H$-invariant} for a subgroup $H$ of $G$ if it is invariant under translation by elements of $H$. Burcroff and Luntzlara \cite{Burcroff} have observed that an $H$-invariant subset $S \subseteq G$ is a minimal complement if and only if the image of $S$ in $G/H$ is a minimal complement. Thus, Theorems~\ref{thm:finite} and \ref{thm:upperbound} have immediate consequences for $H$-invariant subsets when $H$ has finite index. 

\begin{cor}
When ordered by minimal period, almost all periodic subsets of $\mathbb{Z}$ (i.e., subsets that are $n\mathbb{Z}$-invariant for some $n>0$) are not minimal complements. 
\end{cor}

Indeed, the proof of Theorem~\ref{thm:upperbound} shows that almost all $n\mathbb{Z}$-invariant subsets are not minimal complements. An $n\mathbb{Z}$-invariant subset of $\mathbb{Z}$ with minimal period $n/k$ is determined by its intersection with $\{0, 1, \dotsc, n/k - 1\}$, so we see that all but at most $n2^{n/2}$ $n\mathbb{Z}$-invariant subsets have minimal period $n$, which implies the corollary.

In Section~\ref{sec:supplements}, we study ``dual'' problems about maximal supplements.  (For a fixed subset $W$, the problems of finding the smallest complement for $W$ and the largest supplement for $W$ are given by dual integer programs.)  It is tempting to hope that, in analogy with the minimal complement setting, every sufficiently small subset of a finite group is a maximal supplement.  This is not quite the case because every minimal supplement must satisfy a natural ``solidity'' condition: we say that a subset $C$ of an abelian group $G$ is \emph{solid} if $C$ is not properly contained in any set $D$ such that $D-D=C-C$.  It is in fact true that every sufficiently small solid subset is a maximal supplement.
\begin{theorem}\label{thm:suppfinite}
There is an absolute constant $b>0$ such that the following 
holds: Let $G$ be a finite abelian group, 
and let $C \subseteq G$ be a  non-empty solid subset.  
If $$\vert C \vert \leq b \left (\frac{\vert G \vert}{\log \vert G 
\vert} \right)^{1/4},$$ 
then $C$ is a maximal supplement in $G$. 
\end{theorem}

We also show that every finite non-empty solid subset of an infinite abelian group is a maximal supplement; the proof is completely different from (and in many ways simpler than) the proof for the finite setting.
\begin{theorem}\label{thm:suppinfinite}
Let $G$ be an infinite abelian group, and let $C$ be a finite non-empty solid subset of $G$. Then $C$ is a maximal supplement. 
\end{theorem}

\subsection{Open questions}
Our results on minimal complements and maximal supplements raise several natural questions. It would be interesting to find the optimal bound in Theorem~\ref{thm:finite} and understand the asymptotic behavior of $T(n)$. Say that $f(n) = \widetilde{\Theta}(g(n))$ if $f = \Omega(g(n) (\log n)^a)$ and $f(n) = O(g(n) (\log n)^b)$ for some integers $a, b$. 
\begin{conj}
We have that $T(n) = \widetilde{\Theta}(\sqrt{n})$.  More generally, we have that $T(G) = \widetilde{\Theta}(\sqrt{|G|})$.
\end{conj}

Our best constructions of small subsets of finite groups that are not minimal complements rely on the existence of subgroups of certain orders. In particular, our best upper bound on $T(p)$ for $p$ prime is $O(p^{3/4 + \varepsilon})$.
\begin{question}
Is the behavior of $T(n)$ different when $n$ is prime?  How much does $T(G)$ depend on the structure of $G$ in addition to the order of $G$?
\end{question}

We can also ask the analogous question for maximal supplements. Theorem~\ref{thm:suppfinite} shows that non-empty solid subsets of size $O((n/\log n)^{1/4})$ are maximal supplements, but it is not clear that this bound is anywhere near optimal.

\begin{question}
What is the optimal function in Theorem~\ref{thm:suppfinite}?
\end{question}

It is also interesting to consider the case of a random set where each element is included with probability $p$ and determine when such a set 
is a minimal complement. Although the property of being a minimal complement is not monotone, Theorems \ref{thm:finite} 
and \ref{thm:upperbound} (and their proofs) 
show that for small values of $p$ a random set is with high probability a minimal complement, and that for large values of $p$ a random set is with high probability not a minimal complement.

\begin{question}\label{question:randomsets}
How does the probability that a random subset of a finite group $G$ with parameter $p = p(|G|)$ is a minimal complement (or maximal supplement) vary with $p$? Is there a (sharp) threshold? 
\end{question}

\section{Preliminaries}\label{sec:preliminaries}
We begin with a few general results on the behavior of minimal complements under subgroups and quotients, which will be useful in the sequel.  Note first of all that the collection of minimal complements in a fixed group $G$ is invariant under translations by elements of $G$ and applications of automorphisms of $G$.

\begin{lemma}\label{lem:subtobig}
Let $G$ be an abelian group, and let $H$ be a subgroup of $G$. If the subset $C \subseteq H$ is a minimal complement in $H$, then it is also a minimal complement in $G$.
\end{lemma}
\begin{proof}
Suppose $C$ is a minimal complement for $W$ in $H$. Let $K$ contain exactly one element from each coset of $H$. Then $C$ is a minimal complement for $W + K$ in $G$: indeed, $$(W+K)+C=K+(W+C)=K+H=G$$ shows that $C$ is a complement for $W+K$, and minimality follows from the observation that $((W+K)+C) \cap H$ is the translate of $W+C$ by the unique element of $K \cap H$. 
\end{proof}

\begin{lemma}\label{lem:quotient}
Let $C$ be a subset of an abelian group $G$, and let $\pi: G \to H$ be a surjective group homomorphism such that the restriction of $\pi$ to $C$ is injective.  If $\pi(C)$ is a minimal complement (in $H$), then $C$ is a minimal complement (in $G$).
\end{lemma}

\begin{proof}
Let $W \subseteq H$ be a subset for which $\pi(C)$ is a minimal complement.  Then we claim that $C$ is a minimal complement for $\pi^{-1}(W)$.  Indeed, it is clear that $C+\pi^{-1}(W)=G$ (since this set intersects every coset of $\ker(\pi)$ and is invariant under $\ker(\pi)$).  To see that $C$ is minimal, note that any proper subset $D \subset C$ with $D+\pi^{-1}(W)=G$ would give a proper subset $\pi(D) \subset \pi(C)$ that is a minimal complement for $W$.
\end{proof}

We now turn to the possible sizes of ``large'' minimal complements in finite groups.

\begin{prop}\label{prop:size}
Let $G$ be a finite abelian group, and let $W \subseteq G$. If $C$ is a minimal complement for $W$, then $$\vert C \vert \le \vert G \vert \frac{\vert W \vert }{2 \vert W \vert - 1}.$$ In particular, if $C \not= G$, then $\vert C \vert \le 2 \vert G \vert/3$. 
\end{prop}

\begin{proof}
The minimality of $C$ guarantees that for each $c \in C$, there is an element $x(c) \in G$ such that $x(c) \in c + W$ but $x(c) \not \in c' + W$ for any $c' \not= c$.  (Note that these $x(c)$'s are pairwise distinct.)  Let $X=\cup_{c \in C} \{x(c)\}$.  Each of the remaining $|G|-|C|$ elements $g \in G \setminus X$  can be expressed as the sum of an element of $C$ and an element of $W$ in at most $\vert W \vert$ ways. We thus have
$$\vert C \vert \cdot \vert W \vert \le \vert C \vert + (\vert G \vert - \vert C \vert) \vert W \vert,$$
and rearranging gives the result.
\end{proof}

In particular, we see that any $C$ of size $3|G|/5<|C| \leq 2|G|/3$ can be a minimal complement only for a $W$ of size $2$, and it is possible to deduce an explicit structure theorem for minimal complements of this size.  A set $C$ is a minimal complement for some $W=\{w_1, w_2\}$ if and only if it is also a complement for $W-w_1=\{0, w_1-w_2\}$, so we may restrict our attention to sets $W$ containing the identity.  The following characterization of minimal complements of $W=\{0,a\}$ holds in all abelian groups $G$ but is most interesting when $G$ is finite.

\begin{prop}\label{prop:size2}
Let $G$ be an abelian group, and let $W = \{0, a\}$ for some non-zero $a \in G$. The set $C \subseteq G$ is a minimal complement for $W$ if and only if for every $g \in G$, the set $C$ contains at least one of $g$ and $g+a$ but does not contain all of $g$, $g+a$, and $g+2a$.  This condition is equivalent to requiring that the intersection of $C$ with any coset of the cyclic group generated by $a$ neither miss any two consecutive elements nor contain any three consecutive elements.
\end{prop}

\begin{proof}
Let $H$ be the subgroup generated by $a$. Since $W \subseteq H$, we know that $C$ is a minimal complement for $W$ if and only if the following holds for each coset $g+H$: first, $W+(C \cap (g+H))=g+H$; and second, the translate $(C \cap (g+H))-g$ is a minimal complement for $W$ in $H$.  The first condition is satisfied if and only if $C$ contains at least one of $g$ and $g-a$ for every $g\in G$.

For the second condition, consider a complement $D$ for $W$ in $H$.  If there exists $x \in H$ such that $x, x+a, x+2a \in D$, then $D \setminus \{x+a\}$ is also a complement for $W$ in $H$ (which means that $D$ is not minimal).  Suppose, on the contrary, that no such $x$ exists.  Let $d \in D$.  Then either $d+a$ or $d-a$ is not in $D$.  In the former case, the sum $d+a$ is uniquely expressed in $D+W$; in the latter case, the sum $d+0$ is uniquely expressed in $D+W$; so we conclude that $D$ is minimal.  This completes the proof.
\end{proof}

Finally, we mention that short arithmetic progressions in finite groups are minimal complements.  Recall from Proposition~\ref{prop:size} that $G$ does not contain any minimal complements of size strictly between $2|G|/3$ and $|G|$; the following proposition, together with this fact, tells us that the sizes of the minimal complements in $\mathbb{Z}/n\mathbb{Z}$ are exactly $1, 2, \ldots, \lfloor 2n/3 \rfloor, n$.

Before proceeding to the proof, we introduce a useful piece of terminology.  Suppose $C$ is a complement for $W$ in some group $G$.  We say that an element $c \in C$ is \emph{essential} (for this $W$) if $C \setminus \{c\}$ is no longer a complement for $W$.  Clearly, $C$ being a minimal complement for $W$ is equivalent to every element $c \in C$ being essential.

\begin{prop}\label{prop:apcyclic}
Let $G$ be an abelian group of order $n$, and let $C \subseteq G$ be an arithmetic progression of length $k$.  Let $H$ be the subgroup generated by $C$, and write $\vert H \vert = m$. The set $C$ is a minimal complement if and only if
$$k \le \frac{2nm}{2n + m} \quad \text{or} \quad k=m.$$
\end{prop}

\begin{proof}
The case $k=m$ is trivial, so we restrict our attention to $k<m$.  We first treat the special case $G = H = \mathbb{Z}/n\mathbb{Z}$. Since the property of being a minimal complement is invariant under translation and automorphisms, we may assume that $C = \{0, 1, \dotsc, k - 1\}$.

If $k \leq n/2$, then we claim that $C$ is a minimal complement for $W = \{0, k, k+1, \ldots, n-k\}$.  Indeed, $W+C=\mathbb{Z}/n\mathbb{Z}$, and minimality follows from the observation that for every proper subset $D \subset C$, the sumset $D+W$ fails to cover some element of the interval $0,1,\ldots, k-1$.  If instead $n/2<k \leq 2n/3$, then we claim that $C$ is a minimal complement for $W=\{0,k\}$.  Again, it is clear that $W+C=\mathbb{Z}/n\mathbb{Z}$, and minimality is only slightly harder.  For $0 \leq r<n-k$, the element $k+r$ is uniquely represented in $W+C$ as $k+r$, which means that $r$ is essential.  For $n-k \leq r \leq k-1$, the element $r$ is uniquely represented in $W+C$ as $0+r$, which again means that $r$ is essential, so we conclude that $C$ is minimal.  Proposition~\ref{prop:size} tells us that the bound $k \leq 2n/3$ is in fact tight.

We now treat the general case.  Since $H$ is clearly cyclic, we may identify it with $\mathbb{Z}/m\mathbb{Z}$.  By the translation invariance of minimal complements, we may assume that $C=\{0,1,\ldots, k-1\}$ as a subset of $H$. If $k \le 2m/3$, then the previous paragraph ensures that $C$ is a minimal complement in $H$. By Lemma~\ref{lem:subtobig}, $C$ is then also a minimal complement in $G$. We thus restrict our attention to $k>2m/3$. 

We construct $W$ by carefully including two elements of each coset of $H$. Choose representatives $g_1+H, \ldots, g_{n/m}+H$ for $G/H$, so that $G = \cupdot_{i} (g_i + H)$.  Note that $$(W \cap (g_i+H))+C \subseteq g_i+H.$$  Suppose $W_i=W \cap (g_i+H)=\{g_i, g_i+s_i\}$ for some choice of elements $s_i$.  We ensure that $W_i+C=g_i+H$ (i.e., $\{0,s_i\}+C=H$) by picking $s_i$ to satisfy $m-k \leq s_i \leq k$.  Writing $$\{0,s_i\}+C=\{0, 1, \ldots, k-1\} \cup \{s_i, s_i+1, \ldots, s_i+k-1\},$$ we see that every $r \in C$ satisfying $s_i+k-m \leq r \leq s_i-1$ or $k-s_i \leq r \leq m-s_i-1$ is essential.  Now, take $s_i=i(m-k)$ for every $1 \leq i \leq \left\lceil \frac{k}{2(m-k)} \right\rceil$ (where the latter quantity is at most $n/m$ by the hypothesis on the size of $k$ in the statement of the proposition), and set every remaining $s_i=1$ for larger $i$.  It is now easy to see that every $r \in C$ is essential, so we conclude that $C$ is a minimal complement for this $W$, as desired.

Proposition~\ref{prop:lower-construction} (proven in Section~\ref{sec:upper}) gives the converse.
\end{proof}

\section{Small sets are minimal complements}\label{sec:small}

We prove Theorem~\ref{thm:finite} by establishing the following more precise statement. Theorem~\ref{thm:finite} follows from setting $s=\lceil 3/2 \log n \rceil$. In the interest of readability, we have not optimized the bound in the theorem below as being more careful does not give a better exponent in Theorem ~\ref{thm:finite}. The proof is probabilistic and bears some resemblance to the argument in \cite{Al} that every sufficiently large subset of $\mathbb{Z}/p\mathbb{Z}$ is a sumset of the form $A+A$.
\begin{theorem}\label{thm:finiteprecise}
Let $G$ be an abelian group of order $n$, and let $C \subseteq G$ be a non-empty subset of order $k$. 
If there exists a positive integer $s$ such that
\begin{equation}
\label{e11}
\frac{s^2k^3}{n}+ \frac{e^s k^{3s}}{n^{s-1}}
+k\left(\frac{s^2 k^3}{n}\right)^s <1,
\end{equation}
then $C$ is a minimal complement  in $G$.
\end{theorem}

\begin{proof}
Write $C=\{c_1, c_2, \ldots ,c_k\}$, and fix some $s$ such that \eqref{e11} holds.  For each $i \in [k]$, choose $s$ (not necessarily distinct) elements $w_i^{(1)}, w_i^{(2)} , \ldots, w_i^{(s)}$ uniformly and independently at random from $G$.  For each $i \in [k]$ and $p \in [s]$, set $$g_i^{(p)} =w_i^{(p)}+c_i.$$
We will find a subset $W$ for which $C$ is a minimal complement as follows.  First, for each $i$, we will choose one element $w_i^{(p)}$ to include in $W$ (as described below).  Next, we will include in $W$ every $w \in G$ such that $w+C$ does not contain any of the $s$ chosen elements $g_i^{(p)}=w_i^{(p)}+c_i$.  We will have chosen the $w_i^{(p)}$'s so as to make the corresponding $g_i^{(p)}$'s pairwise distinct; this condition will guarantee the minimality of $C$ because $W+(C \setminus \{c_i\})$ does not contain $g_i^{(p)}$.  The main part of the proof is showing that $W+C=G$ with strictly positive probability.

Let $E_1$ be the event that  there are distinct pairs
$(i,p)$ and $(j,q)$ so that $g_i^{(p)}  \in w_j^{(q)} +C$.
There are fewer than $(sk)^2 $ choices for the pairs $(i,p)$ and $(j,q)$,
and for each fixed choice the probability that $g_i^{(p)} \in w_j^{(q)}+C$ is exactly $k/n$, so
\begin{equation}
\label{e21}
\mathbb{P}(E_1) \leq \frac{s^2k^3}{n}.
\end{equation}

Next, let $E_2$ be the event that there exist an element $z \in G$ and $s$ distinct
pairs $(i,p)$ such that each $(g_i^{(p)} -C) \cap (z-C) \neq \emptyset$.
There are 
$$
n \cdot {{sk} \choose s} \leq n (ek)^s
$$
choices for the element $z$ and the $s$ distinct pairs of indices $(i,p)$.
For each such choice, the probability that 
$(g_i^{(p)} -C) \cap (z-C) \neq \emptyset$ for all of the chosen $g_i^{(p)}$'s is at most
$(k^2/n)^s$: the choices of the elements $g_i^{(p)}$ are independent,
and for each fixed $(i,p)$ the probability of such a non-empty intersection
is the probability that $g_i^{(p)} \in z-C+C$, where this set has cardinality smaller than $k^2$. Therefore
\begin{equation}
\label{e22}
\mathbb{P}(E_2)\leq \frac{e^s k^{3s}}{n^{s-1}}.
\end{equation}

Now, order the elements of each translate $z-C$ of $C$ according to the numbering of the elements of $C$, that is, $z-c_1,z-c_2, \ldots, z-c_k$.  Let $E_3$ be the event that there is an index $i \in [k]$ such that the following holds: for each $p \in [s]$, there are $z \in G$, $j \neq i$
in $[k]$, and $q \in [s]$ such that both
\begin{equation}
\label{e231}
\frac{k}{s} < \left| (g_i^{(p)} -C) \cap (z-C)\right| < k
\end{equation}
and 
\begin{equation}
\label{e232}
(g_j^{(q)} -C) ~~\mbox{contains the first element of}~~
(z-C) \setminus (g_i^{(p)} -C).
\end{equation}
There are $k$ possibilities for the index $i$.  Fix such an $i$, and
expose the random elements $g_j^{(q)}$ for all $j \neq i$ and for all
$q$. The union of all such translates $g_j^{(q)} -C$ (call it $Y$) has cardinality 
smaller than $sk^2$. For each index $p$ and each
choice of the element $g=g_i^{(p)}$, let $Z=Z(g)$ denote the set of
elements $z \in G$ for which (\ref{e231}) holds. For each $z \in Z(g)$, let $u=u(g)$ 
be the first element of $z-C$ (in the ordering fixed above) not covered by $g-C$, and let
$U=U(g)$ be the set of all such first elements as $z$ ranges over $Z(g)$.  When $z \in G$ is chosen uniformly at random, the expected size of $(g-C) \cap (z-C)$ is $k^2/n$, so Markov's Inequality tells us that $Z(g)$ (and hence also $U(g)$) has size at most $sk$.  Note that for any $g_1, g_2 \in G$, we have the shifts $Z(g_2)=Z(g_1)+g_2-g_1$ and $U(g_2)=U(g_1)+g_2-g_1$.

The existence of $j \neq i$, $z$ and $q$ satisfying (\ref{e231}) and 
(\ref{e232}) for the index $p$ is equivalent to the condition that 
the random set $U(g_i^{(p)})$ intersects the  set $Y$. Since
$Y$ is a fixed set of size at most $sk^2$ (the elements
$g_j^{(q)}$ already having been exposed) and $U(g_i^{(p)})$ is 
a random shift of a set of size at most $sk$, the probability that
they intersect is at most $s^2k^3/n$. The events corresponding to
distinct indices $p$ are independent, so the probability
of finding such an intersection for all values of $p \in[s]$
is at most $(s^2k^3/n)^s$. As there are $k$ choices for
$i$, we have
\begin{equation}
\label{e23}
\mathbb{P}(E_3) \leq k \left(\frac{s^2k^3}{n} \right)^s.
\end{equation}

It follows from the assumption (\ref{e11}) and the inequalities 
(\ref{e21}), (\ref{e22}), and (\ref{e23}) that with
positive probability none of these three events  holds. Fix a choice of
the random elements $w_i^{(p)}$ for which none of these events holds.
For each $i \in [k]$, pick a $p \in [s]$ such that no choice of 
$z \in G$, $j \neq i$ in $[k]$ and $q \in [s]$ satisfies (\ref{e231}) and (\ref{e232}). (Such a $p$ exists because the event $E_3$ does not occur).  We now define $W$ to contain these $k$ elements $w_i^{(p)}$, together with every element $w \in G$ that is not contained in the union of the chosen $g_i^{(p)}-C$'s (i.e., the translates $w_i^{(p)}+c_i-C$).  Since the event $E_1$ does not occur, we see that each $g_i^{(p)}$ is uniquely expressible in $W+C$ as $w_i^{(p)}+c_i$, which (as discussed above) guarantees the minimality of $C$.  It remains only to show that $W+C=G$.

Let $z \in G$ be an arbitrary element. We have to show that
$W$ intersects $z-C$.  If $z-C=g_i^{(p)} -C$ for one of
our $k$ chosen elements $g_i^{(p)}$, then $$w_i^{(p)} =
g_i^{(p)} -c_i \in g_i^{(p)}-C =z-C,$$ as needed.  Otherwise, no single translate $g_i^{(p)}-C$ covers $z-C$.
We will show that in this case, the $k$ sets $g_i^{(p)}-C$ do not cover
$z-C$. Since the event $E_2$ does not occur, there are at most $s$ sets
$g_i^{(p)}-C$ that intersect $z-C$.  We are done unless there is some particular $g_i^{(p)}-C$ that covers at least $k/s$ elements of $z-C$.   But then, because the event $E_3$ does not occur, we see that there is no $g_j^{(q)}$ such that $g_j^{(q)}-C$ covers the first element of $(z-C) \setminus (g_i^{(p)}-C)$, so $z-C$ is not covered by the translates $g_i^{(p)}-C$.  We conclude that $z \in W+C$.  Since this holds for every $z \in G$, we have $W+C=G$, which completes the proof.
\end{proof}

We now deduce Theorem~\ref{thm:infinite} from Theorem~\ref{thm:finite}. 

\begin{proof}[Proof of Theorem~\ref{thm:infinite}]
Let $G$ be an infinite abelian group, and let $C \subset G$ be a finite subset.  First, suppose $G$ is a torsion group.  Then we can find a finite subgroup $H$ containing $C$ such that $|H|>100 |C|^4$.  By Theorem~\ref{thm:finite}, $C$ is a minimal complement in $H$; Lemma~\ref{lem:subtobig} then tells us that $C$ is also a minimal complement in $G$.

Second, suppose $G$ is not a torsion group, and let $x \in G$ be an element of infinite order.  Then the subgroup $H$ generated by the set $C \cup \{x\}$ is finitely generated and has infinite order; write $H=K \times \mathbb{Z}^r$, where $K$ is a finite group and $r \geq 1$.  Choose some $M>100 |C|^4$ such that $C$ is contained in the ``cube'' $K \times [-M,M)^r$, and let $\pi: H \to K \times (\mathbb{Z}/2M \mathbb{Z})^r$ be the canonical projection map.  Note that $\pi$ is injective on $C$.  By Theorem~\ref{thm:finite}, $\pi(C)$ is a minimal complement in $K \times (\mathbb{Z}/2M \mathbb{Z})^r$; by Lemma~\ref{lem:quotient}, $C$ is a minimal complement in $H$; by  Lemma~\ref{lem:subtobig}, $C$ is a minimal complement in $G$, as desired.
\end{proof}

We remark that one can directly prove a slightly stronger version of Theorem~\ref{thm:infinite} by performing the construction of the proof of Theorem~\ref{thm:finite} in the infinite group and using transfinite induction. Taking $s = 1$, an examination of the proof shows that one can iteratively construct $W$ by adding  $c_i + w_i$ for each $c_i \in C$, and then adding all other elements that do not destroy the minimality of $C$. When we try to find the next $w_i$ to add to $W$, we make sure that the events $E_1, E_2,$ and $E_3$ do not occur. This will be true if the chosen $w_i$ is not contained in a certain union of translates of $C - C$ (determined by the elements that we already added to $W$), and the number of such translates is less than $\vert C \vert$ or is finite. We can therefore find an appropriate $w_i$ as long as $S + C - C \neq G$, for certain subsets $S$, with the cardinality of $S$ equal to the number of elements of $W$ already chosen. Therefore this procedure will succeed as long as $S + C - C$ is a proper subset of $G$ whenever $S$ has cardinality strictly smaller than that of $C$ or is finite. See also the proof of Theorem~\ref{thm:suppinfinite}.

\section{Many large sets are not minimal complements}\label{sec:upper}
We now turn to the problem of finding upper bounds on $T(G)$. Our first result shows that $T(n) = O(\sqrt{n})$ under certain divisibility conditions on $n$. We then proceed with the probabilistic proof of Theorem~\ref{thm:upperbound}, which extends to all values of $n$ at the cost of increasing the exponent from $1/2$ to $3/4+\varepsilon$.

\begin{prop}\label{prop:lower-construction}
Let $G$ be an abelian group of order $n$ with a subgroup $H$ of order $m$, and let $k$ be an integer satisfying $$\frac{2nm}{m+2n}<k<m.$$
Then no subset $C \subset H$ of size $k$ is a minimal complement in $G$.
\end{prop}

\begin{proof}
Assume (for the sake of contradiction) that $C$ is a minimal complement for some $W$ in $G$.  Since $C+W=G$ and $C$ is properly contained in $H$, we see that $W$ must contain at least two elements of each of the $n/m$ cosets of $H$.  Consider a coset $g+H$, and let $w_1,w_2$ be distinct elements of $W \cap (g+H)$.  The translates $w_1+C$ and $w_2+C$ each have size $k$ and are contained in $g + H$.  In the sum $\{w_1, w_2\}+C$, each element of $g+H$ is represented at most twice, so (since there are exactly $2k$ total sums) at least $2k-m$ elements are represented twice.  It follows that in the sum
$$(W \cap (g+H))+C=g+H,$$
there are at most $2m-2k$ elements that are uniquely represented.  Since this is true of each of the $n/m$ cosets of $H$, there are at most
$$\frac{2n(m-k)}{m}<k$$
elements of $G$ that are uniquely represented in $W+C$.  Thus, fewer than $k$ elements of $C$ are essential, which contradicts the minimality of $C$.
\end{proof}

\begin{cor}
We have that $\liminf_{n \to \infty} \frac{T(n)}{\sqrt{n}} \le \sqrt{2}$. 
\end{cor}
\begin{proof}
Let $G = \mathbb{Z}/n\mathbb{Z}$, where $n = (k + 1)(\lceil k/2 \rceil - 1)$. Then the element $\lceil k/2 \rceil - 1$ generates a subgroup of order $m=k+1$, and Proposition~\ref{prop:lower-construction} implies that $T(\mathbb{Z}/n\mathbb{Z}) \le k$. 
\end{proof}

We now transition to the proof of Theorem~\ref{thm:upperbound}.  Let $G$ be an abelian group of order $n$, and let $A$ be a random subset
of $G$ obtained by including each $g \in G$ independently with probability $p=n^{-1/4+\varepsilon}$.
In the following discussion, we say that an event holds \emph{with high probability} (\emph{whp}) if it holds with probability tending to $1$ as $n$ tends to infinity.
\begin{lemma}
\label{l21}
Let $G$ and $A$ be as above. Then whp
$|A|=(1+o(1)) n^{3/4+\varepsilon}$ and
every complement $B$ for $A$ (that is, a subset $B \subseteq G$ with $A+B=G$) satisfies
$$
|B| > \frac{1}{4p} \log n=\frac{1}{4} n^{1/4-\varepsilon} \log n.
$$
\end{lemma}
\begin{proof}
The fact that $|A|=(1+o(1)) n^{3/4+\varepsilon}$ whp is trivial.
Fix a subset $B$ with $m=\frac{1}{4p} \log n$ elements.
There is a subset $T \subseteq G$ with at least $n/m^2=\tilde{\Theta}(n^{1/2+2\varepsilon})$
elements such that $t_1-B$ and $t_2-B$ are disjoint for all distinct elements $t_1,t_2 \in T$ (i.e., $(T-T) \cap (B-B)=\{0\}$).   For each fixed
$t \in T$, the probability of $t-B$ being disjoint from $A$ is precisely
$(1-p)^m \geq n^{-1/4+o(1)}$. Thus, the probability that
$t \in B+A$ (i.e., $t-B$ intersects $A$) is at most $1-n^{-1/4+o(1)}$.  Since the
events corresponding to distinct elements of $T$ are mutually independent, the choice of $T$ ensures that $T \subseteq A+B$ with probability at most
$$
(1-n^{-1/4+o(1)})^{|T|} \leq e^{-n^{1/4+\varepsilon}}.
$$
This is clearly an upper bound for the probability that $A+B=G$.
As there are only ${n \choose m} 
=e^{\tilde{\Theta}(n^{1/4-\varepsilon})}$ sets $B$ of size $m$, it follows that whp there is no such
$B$ satisfying $A+B=G$.
\end{proof}
\begin{lemma}
\label{l22}
Let $G$ and $A$ be as above. Then whp every subset
$D \subseteq G$ of size
$$
|D| \geq (1-\varepsilon)\frac{1}{4p} \log n=(1-\varepsilon)
\frac{1}{4} n^{1/4-\varepsilon} \log n
$$ satisfies $|G\setminus (A+D)| \leq n^{3/4+\varepsilon/2}$.
\end{lemma}
\begin{proof}
It clearly suffices to prove the statement for every  $D$ of cardinality
exactly $(1-\varepsilon)\frac{1}{4p} \log n$. 
Fix such a set $D$. Let $H$ be the Cayley 
graph of $G$ with respect to $(D-D) \setminus \{0\}$: the graph on the vertex set $G$ where distinct $g_1,g_2 \in G$ are adjacent if and only if there are 
$d_1,d_2 \in D$
such that $g_1=g_2+d_1-d_2$ (i.e., the sets $g_1-D$ and $g_2-D$
intersect).  The maximum degree of this graph is smaller than 
$|D|^2$. Therefore, by a well known theorem of 
Hajnal and Szemer\'edi \cite{HS} (which is
convenient to use here but can also be avoided), the graph $H$ 
has a proper coloring with at most $|D|^2=\tilde{\Theta}(n^{1/2-2\varepsilon})$
colors where the color classes differ in size by at most $1$.

For each color class $C$, the number of elements $t \in C$ that are not
covered by $A+D$ is a binomial random variable with parameters
$|C| = \tilde{\Theta}(n^{1/2+2\varepsilon})$ and 
$q=(1-p)^{|D|} =n^{-1/4(1-\varepsilon+o(1))}$.
By the standard estimates for binomial distributions (see, e.g., 
\cite{AS}, Appendix A), the probability the 
value of this random variable exceeds its expectation by a factor of
$2$ is exponentially small in the expectation, which in turn is larger than
$n^{1/4+\varepsilon}$. Therefore, whp this does not happen
for any of the color classes and any choice of $D$, as there are only
$e^{\tilde{\Theta}(n^{1/4-\varepsilon})}$ 
choices for $D$.  But this means that
whp $|G\setminus (A+D)| \leq n^{3/4+\varepsilon/2}$ holds for every such $D$.
\end{proof}
Using these two lemmas, we proceed with the proof of
Theorem~\ref{thm:upperbound}.
\begin{proof}[Proof of Theorem~\ref{thm:upperbound}]
Let $A$ be a set of size $(1+o(1))n^{3/4+\varepsilon}$ 
satisfying the conclusions of
Lemma \ref{l21} and Lemma \ref{l22}. We claim that $A$ is not a minimal
complement.  Assume (for contradiction) that $A$ is a minimal complement for some subset $B \subseteq G$.
Lemma \ref{l21} tells us that $|B| \geq \frac{1}{4p} \log n$. Partition $B$ 
arbitrarily into $t =\lceil 1/\varepsilon \rceil$ (pairwise disjoint) parts
$B=B_1 \cupdot B_2 \cupdot \ldots \cupdot B_t$ of equal cardinalities (up to a difference of $1$).
For each $i \in [t]$, let $D_i=B\setminus B_i$. Then each $D_i$ has
cardinality
$$
|D_i| \geq(1-\varepsilon) |B| \geq (1-\varepsilon) \frac{1}{4p} \log n. 
$$ Since $A$
satisfies the conclusion of Lemma \ref{l22}, each of the $t$ sets
$A+D_i$ covers all but at most
$n^{3/4+\varepsilon/2}$ of the elements of $G$. So in the sum $A+B$,
all but at most $t n^{3/4+\varepsilon/2}<|A|$ of the elements of $G$ 
are covered by every set $A+D_i$ and are hence covered at
least twice.  In particular, there are strictly fewer than $|A|$ elements that are covered exactly once, so not every element of $A$ can be essential, and $A$ is not in fact minimal as a complement for $B$.  This contradiction completes the proof.
\end{proof}

\section{Maximal supplements}\label{sec:supplements}

Problems about large supplements are ``dual'' to problems about small complements in a natural way: given a subset $W$ of a finite abelian group, the problem of finding the minimum size of a complement for $W$ and the problem of finding the maximum size of a supplement for $W$ are dual as integer programs.  Of course, not all minimal complements have the minimum size, and not all maximal complements have the maximum size, but this connection nonetheless motivates studying the two problems together.

We can characterize being a maximal supplement as follows.  $C$ is a supplement to $W$ if and only if $c_1+w_1=c_2+w_2$ has only the trivial solutions for $c_1, c_2 \in C$, $w_1, w_2 \in W$.  In particular, this condition is equivalent to $(C-C) \cap (W-W)=\{0\}$.  The maximality condition on $C$ is that for every $d \in G \setminus C$, the translate $d-C$ intersects $W-W$ nontrivially, i.e., $d \in C+W-W$.  Since $C \subseteq C+W-W$ trivially, this maximality condition can be expressed as $C+W-W=G$.  So, putting everything together, we have that $C$ is a maximal supplement in $G$ if and only if there is a subset $W$ satisfying
$$(C-C) \cap (W-W)=\{0\} \quad \text{and} \quad C+W-W=G.$$

One might ask if inverse results about minimal complements carry over to the setting of maximal supplements.  In particular, one might ask if every finite subset of $\mathbb{Z}$ is a maximal supplement and if every sufficiently small subset of a finite abelian group is a maximal supplement.  The answer to each of these questions is ``no'', for  a simple reason.  Recall that a subset $C$ of an abelian group $G$ is \emph{solid} if $C$ is not properly contained in any set $D$ such that $D-D=C-C$. We show that every maximal supplement is solid. 

\begin{prop}\label{prop:solid}
Let $G$ be an abelian group.  If the subset $C \subseteq G$ is a maximal supplement, then $C$ must be solid.
\end{prop}

\begin{proof}
Assume (for the sake of contradiction) that $C$ is a maximal supplement to some $W$ but $C$ is not solid, and let $D$ be a set properly containing $C$ such that $D-D=C-C$.  Then
$$(D-D) \cap (W-W)=(C-C) \cap (W-W)=\{0\}$$
shows that $D$ is also a supplement to $W$, which contradicts the maximality of $C$.
\end{proof}

Now, we may ask if solidity is sufficient for a set to be a maximal supplement.  For finite subsets of  infinite groups, the answer is ``yes''. We first require a lemma that establishes a sufficient condition for $C$ to be a maximal supplement for a particular $W$.

\begin{lemma}\label{lem:supp-completion}
Let $G$ be an abelian group, and let $C \subseteq G$ be a non-empty solid subset.  If the subset $W \subseteq G$ satisfies
$$G \setminus (W-W)=(C-C)\setminus \{0\},$$
then $C$ is a maximal supplement for $W$.
\end{lemma}

\begin{proof}
It is immediate that $$(C-C) \cap (W-W)=\{0\}.$$
Suppose (for the sake of contradiction) that $C+W-W \neq G$, and let $x \in G \setminus (C+W-W)$.  Then $x-C$ is disjoint from $W-W$, which means that
$$x-C \subseteq (C-C)\setminus \{0\}.$$
Note that $x \notin C$ since $x-C$ cannot contain the element $0$.  Now, the set $D=C \cup \{x\}$ witnesses that $C$ is not solid:
$$(C \cup \{x\})-(C \cup \{x\}) \subseteq (C-C) \cup (x-C)\cup (C-x) \subseteq C-C.$$
This contradiction lets us conclude that in fact $C$ is a maximal supplement to $W$.
\end{proof}

We now prove the promised result about finite subsets of infinite groups.

\begin{reptheorem}{thm:suppinfinite}
Let $G$ be an infinite abelian group, and let $C \subset G$ be a finite non-empty solid subset. Then $C$ is a maximal supplement in $G$. 
\end{reptheorem}

\begin{proof}
We construct a set $W$ for which $C$ is a maximal supplement.  Using the well-ordering principle,  fix a bijection of $(G \setminus (C - C)) \cup \{0\}$ with the minimal ordinal of cardinality $\vert G \vert$. This gives an ordering $\{g_\lambda\}_{\lambda \in \Lambda}$ on $(G \setminus (C - C)) \cup \{0\}$ where $0$ is the minimal element of $\Lambda$ and the set $\{\alpha: \enskip \alpha < \lambda\}$ has cardinality strictly smaller than that of $G$ for every $\lambda \in \Lambda$.  We may take $g_0=0$.  We now use transfinite induction to construct a sequence of (increasing) nested subsets $\{W_{\lambda}\}_{\lambda \in \Lambda}$ such that each $W_\lambda$ satisfies
$$(C-C) \cap (W_\lambda-W_\lambda)=\{0\} \quad \text{and}  \quad g_{\lambda'} \in W_\lambda-W_\lambda  \text{ for all }\lambda' < \lambda.$$ 
Then it is clear by Lemma~\ref{lem:supp-completion} that $C$ is a maximal supplement for
$$W=\bigcup_{\lambda \in \Lambda} W_\lambda,$$
since $G \setminus (W - W) = (C - C) \setminus \{0\}$.  For the base cases, set $W_0=W_1=\{0\}$.  For a limit element $\alpha$, we set $W_{\alpha} = \cup_{\lambda < \alpha} W_{\lambda}$.  For a successor element $\lambda+1$, suppose we already have $W_\lambda$.  We wish to find some $x \in G$ such that we may ``safely'' define $$W_{\lambda+1}=W_\lambda \cup \{x, x+g_{\lambda}\};$$
in order to ensure that $W_{\lambda+1}-W_{\lambda+1}$ intersects $C-C$ only at $0$, it suffices to check that
$W_\lambda-\{x, x+g_{\lambda}\}=W_\lambda -\{0,g_{\lambda}\}-\{x\}$ is disjoint from $C-C$ (since $C-C$ is symmetric).  Indeed, the induction hypothesis ensures that $W_\lambda-W_\lambda$ intersects $C-C$ only at $0$, and the elements $\pm g_{\lambda}$ of $W_{\lambda+1}-W_{\lambda+1}$ are not in $C-C$ by assumption.  So we must find $x$ such that
$$x \notin W_\lambda-\{0,g_{\lambda}\}-C+C.$$
Since $C - C$ is finite, the set $W_\lambda-\{0,g_{\lambda}\}-C+C$ is finite if $W_{\lambda}$ is finite and has cardinality equal to $W_{\lambda}$ if $W_{\lambda}$ is infinite. As the cardinality of $W_{\lambda}$ is strictly smaller than that of $G$, there always exists some such choice of $x$.
\end{proof}

This argument actually gives the following slightly more general statement.

\begin{theorem}
Let $G$ be an infinite abelian group, and let $C \subset G$ be a non-empty solid subset such that $G \neq S + C - C$ whenever $\vert S \vert < \vert G \vert$. Then $C$ is a maximal supplement in $G$.
\end{theorem}

As was the case for additive complements, the finite group setting is more complex.  Our next aim is showing that every sufficiently small solid subset of a finite abelian group is a maximal supplement.  We again apply Lemma~\ref{lem:supp-completion}.  An old result of the first author \cite{Al} shows that every sufficiently large symmetric subset containing $0$ of a finite abelian group is a difference set of the form $A-A$.  The following appears as Theorem 4.1 in \cite{Al}.  (As discussed in that paper, the bound is essentially tight.)

\begin{theorem}[Alon \cite{Al}]\label{thm:diff}
There exists an absolute constant $c>0$ such that the following holds: Let $G$ be an abelian group of order $n$, and let $V=-V$ be a symmetric subset of $G$ containing $0$.  If
$$|V| \geq n-c\sqrt{n/\log n},$$
then there is some subset $A \subseteq G$ such that $V=A-A$.
\end{theorem}

Combining these two pieces shows that every sufficiently small solid subset is a maximal supplement.

\begin{proof}[Proof of Theorem~\ref{thm:suppfinite}]
We have $|C-C| \leq b^2 \sqrt{n/\log n}$, so that $V=(G \setminus (C-C)) \cup \{0\}$ is a symmetric subset containing $0$ and $|V| \geq n-b^2 \sqrt{n/\log n}$.  By Theorem~\ref{thm:diff}, a sufficiently small choice of $b$ guarantees the existence of some $W \subseteq G$ such that $V=W-W$.  Since this $W$ satisfies the conditions of Lemma~\ref{lem:supp-completion}, we conclude that $C$ is a maximal supplement for $W$.
\end{proof}

As in the proof of Theorem~\ref{thm:infinite}, it is possible to deduce Theorem~\ref{thm:suppinfinite} from Theorem~\ref{thm:suppfinite} by showing that the property of being a maximal supplement lifts from certain subquotients. 

\section*{Acknowledgments}
The first author is supported in part by NSF grant DMS-1855464, ISF grant 281/14, and BSF grant 2018267. We thank the referee for several helpful comments.


\end{document}